\documentclass[12pt]{article}
\usepackage{amsfonts}
\usepackage{amsthm}
\usepackage{amsmath}
\usepackage{amssymb}
\usepackage{graphicx}
\usepackage{booktabs}
\usepackage{rotating}
\usepackage{array}
\usepackage{mathtools}
\usepackage{enumerate}
\usepackage{chemfig}
\usepackage{tikz}

\numberwithin{equation}{section}
\numberwithin{table}{section}
\numberwithin{figure}{section}

\newtheorem{theorem}{Theorem}[section]
\newtheorem{lemma}[theorem]{Lemma}
\newtheorem{proposition}[theorem]{Proposition}
\newtheorem{corollary}[theorem]{Corollary}
\newtheorem{definition}[theorem]{Definition}
\newtheorem{example}[theorem]{Example}
\newtheorem*{theorem*}{Theorem}

\renewcommand{\H}{{\mathbb H}}      
\newcommand{\C}{{\mathbb C}}      
\newcommand{\R}{{\mathbb R}}      
\newcommand{\Z}{{\mathbb Z}}      
      
\newcommand{\N}{{\mathbb N}}      
\newcommand{\CP}{{\mathbb{CP}}}      
\newcommand{\ol}{\overline}
\renewcommand{\vec}[1]{\mathbf{#1}}

\title{The twistor discriminant locus of the Fermat cubic}
\author{John Armstrong}
\date{April 2014}

\begin{document}

\maketitle

\begin{abstract}
We consider the discriminant locus of the Fermat cubic under the twistor fibration $\CP^3 \longrightarrow S^4$. We show that it has a conformal symmetry group of order $72$ and use this to identify its topology.
\end{abstract}

\section{Introduction}

Orientation preserving conformal maps of $S^4 =
\R^4 \cup \{\infty\}$ to itself give rise, via the twistor construction, to
projective transformations of the twistor space $\CP^3$.
So, in the spirit of Klein's Erlangen program, one can choose to
consider the classical geometry of $\CP^3$ modulo the projective
transformations or the twistor geometry of $\CP^3$ modulo the
orientation preserving conformal maps of $S^4$. In particular instead
of classifying algebraic surfaces up to projective
transformation, one can instead attempt to classify them up to
orientation preserving conformal transformation.

For the sake of brevity, in the rest of this paper, we will redefine conformal
to mean orientation and angle preserving rather than merely
angle preserving.

The defining polynomial equation of a degree $d$ complex surface
$\Sigma$ in $\CP^3$ reduces to a polynomial in one variable of
degree $d$ on each fibre of the twistor fibration. So when
restricted to $\Sigma$, the twistor fibration gives a
$d$-sheeted cover $S^4$. The branch locus is given by the points
where the discriminant of the polynomial on each fibre vanishes,
hence this is called the discriminant locus.

The topology of the discriminant locus is an invariant of
$\Sigma$ under the conformal group.  Thus understanding the
topology of the discriminant locus is a natural question when
classifying surfaces modulo the conformal group.

Related to the study of the discriminant locus is the study of twistor lines.
If a fibre of the twistor projection is contained in $\Sigma$,
then the polynomial on that fibre completely vanishes. Such a fibre is called a twistor line of $\Sigma$. Its projection onto $S^4$ is a singular point of the discriminant locus.

The classification of non-singular degree $2$ surfaces under
conformal transformations was completed in
\cite{salamonViaclovsky}. The classification shows that for
non-singular degree $2$ surfaces there are always either
$0$, $1$, $2$ or $\infty$ fibres of the twistor fibration lying on
the surface and that the topology of the discriminant locus is
completely determined by the number of number of twistor lines.

A study of twistor lines on cubic surfaces was made in \cite{armstrongSalamon}.
The aim of this paper is to consider the topology of the
discriminant locus on such surfaces and, in particular, to calculate
the topology of the discriminant locus for the Fermat cubic. This
is the cubic surface given by the equation
\begin{equation} z_1^3 + z_2^3 + z_3^3 + z_4^3 = 0.
\label{fermatCubic}
\end{equation}
A key ingredient in the proof is the observation that the discriminant locus has a suprisingly large conformal symmetry group of order $72$. By contrast, the Fermat cubic itself has only $6$ conformal symmetries.

The question of computing the topology of the discriminant locus in degrees $d\geq2$ was raised in \cite{povero}. The topology of a projectively, but not conformally, equivalent cubic surface was computed informally in \cite{armstrongSalamonSigma}. However, the argument in \cite{armstrongSalamonSigma} is not fully rigorous since it depends upon visual examination of the discriminant locus. This paper contains the first fully rigorous calculation of the discriminant locus of an irreducible surface of degree $d\geq2$.

The structure of the remainder of the paper is as follows.

We begin by clarifying our notation and reviewing the twistor fibration in Section \ref{reviewSection}. The reader should consult \cite{atiyah} for further background on the Twistor fibration.

In Section \ref{algebraicSection} we prove the basic algebraic facts about the discriminant locus. We show in particular how to choose coordinates that significantly reduce the algebraic complexity of the discriminant locus when it has a twistor line.

In Section \ref{topologicalSection} we prove some general facts
about the topology of the discriminant locus. In particular we compute its Euler characteristic, dimension and orientability and consider its (non)-smoothness properties.

In Section \ref{singularSurfacesSection} we review the classification of surfaces with isolated singularities. The purpose of this is to establish a graphical notation which we can use to unambiguously describe the topology of a singular surface.

In Section \ref{fermatSection} we compute the topology
of the discriminant locus of the Fermat cubic using a visual inspection. The aim
of the remaining two sections is to provide a rigorous justification for this computation.
In Section \ref{symmetrySection} we analyse the symmetries of the discriminant locus 
which allows us to significantly simplify the problem. With these simplifications in place,
in Section \ref{cylindricalDecompositionSection} 
we are able to apply the cylindrical algebraic decomposition algorithm
to complete the proof.

\section{Review of the twistor projection $\CP^3 \rightarrow S^4$}
\label{reviewSection}

Let us quickly review the setup.

Define two equivalence relations on $(\H \times \H )\setminus \{0\}$, denoted $\sim_\C$ and
$\sim_\H$, by
\[ (q_1,q_2)\sim_\H(\lambda q_1, \lambda q_2) \quad  \lambda \in \H \setminus \{ 0 \} \] 
\[ (q_1,q_2)\sim_\C(\lambda q_1, \lambda q_2) \quad  \lambda \in \C \setminus \{ 0 \} \] 
$((\H \times \H) \setminus \{0\})/ \sim_\C \cong \CP^3$ and $((\H \times \H) \setminus \{0\})/ \sim_\H
\cong {\mathbb HP}^1 \cong \R^4 \cup \{ \infty \} \cong S^4$.
An explicit map from $\CP^3$ to $((\H \times \H) \setminus \{0\})/\sim_\C$ is given by
$[z_1,z_2,z_3,z_4] \rightarrow [z_1+z_2 j, z_3 + z_4j]_{\sim_\C}$.

The map $\pi:[q_1,q_2]_{\sim_\C} \longrightarrow
[q_1,q_2]_{\sim_\H}$ is the twistor fibration.

Left multiplication by the quaternion $j$ induces an
antiholomorphic involution of $\CP^3$ given by $[q_1,q_2]_{\sim_\C} \mapsto [j q_1,j
q_2]_{\sim_\C}$. We will call this involution $j$, it should be
clear from the context whether we are referring to the quaternion $j$
or to this map. The map $j$ acts on each fibre of $\pi$ as the antipodal map.

The group of conformal transformations of $S^4$ is given by the
quaternionic M\"obius transformations acting on the quaternionic
projective space on the right. That is for $4$ quaternions $a$,
$b$, $c$, $d$ we define the associated
M\"obius transformation by:
\[ [q_1. q_2]_\sim \rightarrow [ q_1 a + q_2 b, q_1 c + q_2
d]_\sim. \]
This can be viewed either as a conformal transformation of $S^4$
or as a fibre preserving holomorphic map of $\CP^3$. Thinking of
$S^4$ as $\H \cup \{ \infty \}$ we can write this map as:
\[ q \rightarrow (qc + d)^{-1}(qa + b). \]

If we use the notation $q=q_1 + q_2j$ to split a quaternion into two complex
components $q_1$ and $q_2$ then we can explicitly write the projective
transformation associated with the M\"obius transformation $(qc + d)^{-1}(qa + b)$
as:
\[ \left( \begin{array}{cccc}
a_1 & -\ol{a_2} & b_1 & -\ol{b_2} \\
a_2 & \ol{a_1} & b_2 & \ol{b_1} \\
c_1 & -\ol{c_2} & d_1 & -\ol{d_2} \\
c_2 & \ol{c_1} & d_2 & \ol{d_1}
\end{array} \right)
\]
Thus a projective transformation arises from a conformal transformation of $S^4$
if and only if it has the complex conjugation symmetries shown in the matrix
above.

Let us give formal definitions of the central notions in this paper:

\begin{definition} The \em{discriminant locus} ${\cal D}$ of a degree $d$
complex surface $\Sigma$ in $\CP^3$ is the set:
\[ \{ x \in S^4 : \#( \pi^{-1}(x) ) < d \}. \]
\end{definition}
\begin{definition}
A \em{twistor line} of $\Sigma$ is a fibre of $\pi$ which lies entirely within $\Sigma$.
\end{definition}
\begin{definition}
When $d=3$ we define the set of \em{triple points} of $\Sigma$ to be the set:
\[ \{ x \in S^4 : \#( \pi^{-1}(x) ) = 1 \} \]
\end{definition}

\section{Algebraic properties of the discriminant locus}
\label{algebraicSection}

We can now state the main algebraic results about the
discriminant locus.

\begin{proposition}
With the exception of a possible point at $\infty$, the
discriminant locus of a general degree $d$ surface is described by
the zero set of a complex valued polynomial of degree $2d(d-1)$ in
the real coordinates $\vec{x}=(x_1,x_2,x_3,x_4)$ for $S^4 = \R^4
\cup \{ \infty \}$.
If the surface contains a twistor line over $\infty$ this simplifies to a polynomial of degree $2(d-1)^2$.
\label{algebraicProposition}
\end{proposition}
\begin{proof}
Let the surface be defined by the equation
$f(z_1,z_2,z_3,z_4)=0$ for a homogeneous polynomial $f$ of degree $d$.

Define a map $p_1:\R^4 \longrightarrow \CP^3$ by
$p_1(x_1,x_2,x_3,x_4) = [x_1 + x_2 i +x_3 j + x_4
k,1]_{\sim_\C}$.
Define $p_2=j \circ p_1$.

Define $\psi:\R^4 \times \CP^1 \longrightarrow \CP^3$ by
$\psi(\vec{x},[\lambda_1,\lambda_2]) = \lambda_1 p_1(\vec{x}) +
\lambda_2 p_2(\vec{x})$. Thinking
of $S^4$ as $\R^4 \cup \{ \infty \}$, one sees that $\psi$ gives
a trivialization of the twistor fibration away from the point
$\infty$.

Viewing $\R^4$ as $\C \times \C$ we can define $p_1(w_1,w_2)$
for complex numbers $w_1=x_1+i x_2$ and
$w_2=x_3 + i x_4$. So $f \circ p_1$ is an inhomogeneous
polynomial in $w_1$, $w_2$ of degree at most $d$. Its zero set
defines a complex affine curve ${\cal C}_1$. This affine curve is
mapped by $p_1$ to the intersection of the surface $f=0$ and
the affine plane in $\CP^3$ defined by the conditions $z_4=0$
and $z_3 \neq 0$. $p_1({\cal C}_1)$ thus lies inside
the degree $d$ curve ${\cal C}_2$ defined by the intersection of $f=0$ and the plane $z_4=0$.

If ${\cal C}_2$ is irreducible, we may conclude that $f \circ
p_1$ is of degree exactly $d$. However if ${\cal C}_2$ contains
the line defined by the conditions $z_3=z_4=0$ then $c \circ
p_1$ will be of degree at most $d-1$. We deduce that $f \circ p_1$ will be of degree less than $d$ if the surface $f=0$ contains the line $z_3=z_4=0$. Equivalently $f \circ p_1$ will be of degree $\leqslant d-1$ if and only if the surface $f=0$ contains the twistor line over $\infty$.

Using the same argument with a different identification of
$\R^4$ and $\C \times \C$ gives the same result for $f \circ
p_2$. Indeed one has this result for any function $f \circ
(\lambda_1 p_1 + \lambda_2 p_2)$ for complex numbers $\{
\lambda_1, \lambda_2 \} \in \C^2 \setminus \{0\}$.

We deduce that the function $f(\vec{x},\lambda) = c( p_1(\vec{x}) +
\lambda p_2(\vec{x}))$ is a degree d polynomial in $\lambda$ with coefficients of degree $d$ in $\vec{x}$ (of degree less than $d$ if $f=0$ contains the twistor line over $\infty$). To see this write
$f(\vec{x},\lambda)=a_d(\vec{x}) \lambda^d+ a_{d-1}(\vec{x}) \lambda^{d-1} + \ldots
a_1(\vec{x}) \lambda + a_0(\vec{x}) $, for some
polynomials $a_i$. Inserting the values
$\lambda=0,1,2,\ldots,d$ in this expression gives $d+1$ linearly
independent expressions in the $a_i$ in terms of
the degree $d$ (or $d-1$) polynomials $f(\vec{x},\omega)$. Solving
these equations allows us to express $a_i$ as
degree $d$ (or $d-1$) polynomials.

Recall that the discriminant of a polynomial of degree $d$ is a polynomial of degree $2(d-1)$ in the coefficients. The
result now follows.
\end{proof}

For cubic surfaces, the discriminant locus is of degree $12$ or, if their is a twistor line, of degree $8$.

We will also be interested in the number of triple points of a cubic surface. If we recall that the cubic $a \omega^3 + b \omega^2 + c \omega + d$ has a triple root if and only if $b^2 - 3 a c=0$ and $9 a d-b c = 0$ we find from the argument above:
\begin{proposition}
The set of triple points of a cubic surface is the zero set of two
complex polynomials of degree $8$
in $\vec{x}$. If there is a twistor line at $\infty$ these
simplify to polynomials of degree $4$.
\end{proposition}

We make some remarks.
\begin{enumerate}[(i)]
\item The discriminant locus is defined by the zero set of a
complex valued polynomial. Taking real and imaginary parts of
this polynomial, the discriminant locus can be defined by
intersection of the zero sets of two real valued polynomials.
\item It is straightforward to perform this calculation
explicitly with a computer algebra system to find the explicit
polynomials. When one notices that a general degree 12 polynomial
in 4 variables has $1820$ coefficients and a general degree 8 polynomial in 4 variables has $495$ coefficients then the need for computer algebra in calculations becomes obvious.
\item The drop in the degree when one has a twistor line at
infinity should be seen as a substantial simplification. Many of
the known algorithms in computational geometry have
computational complexity of the approximate form $p(d)^{c^{n-1}}$ where $p$
is a polynomial in the degree $d$ of the polynomials under consideration, $n$ is the
dimension and $c$ is some constant. Key examples of such algorithms
are the computation of Gr\"obner bases (see \cite{mayrAndMeyer, dube}) and of
cylindrical algebraic decompositions (see \cite{basu}).
\end{enumerate}

\section{The topology of the discriminant locus}
\label{topologicalSection}

Let us make some general observations about the topology of the
discriminant locus.

\begin{definition}
The {\em double locus} of ${\cal D}$ denoted ${\cal D}^\prime$ is the set of points $x$ in ${\cal D}$ where $f_{\restriction \pi^{-1}(x)}$ has a unique double point. 
\label{smoothPointDefinition}
\end{definition}

\begin{proposition}
The discriminant locus ${\cal D}$ of a degree $d$ complex
surface $\Sigma$ defined by the equation $f=0$, with $f$ a square
free polynomial, is compact and of real dimension $\leqslant 2$. The set ${\cal D}^\prime$ is a smooth orientable real surface.
\label{orientabilityProposition}
\end{proposition}
\begin{proof}
The discriminant locus can locally be written as the zero set of
a complex valued polynomial in the coordinates $\vec{x}$.
Therefore it is compact.

Write $\Sigma=\Sigma^1 \cup \Sigma^2$ where $\Sigma^1$ consists of the smooth points of $\Sigma$ and $\Sigma^2$ contains the non-smooth points. Since $f$ is square free, $\Sigma^2$ has complex dimension of one or less.

A point $x_0$ lies in ${\cal D} \setminus \pi(\Sigma^2)$ if and only if at some point in $z \in \pi^{-1}(x) \cap \Sigma^1$, the
tangent space $T_z \Sigma$ contains the tangent of the twistor fibre at $z$. Therefore the restriction of $\pi$ to $\Sigma^1$ has
rank $2$ at $z$. The Sard--Federer theorem then implies that the
Hausdorff dimension of ${\cal D} \setminus \pi(\Sigma^2)$
is $\leqslant 2$. Since we are working in the real algebraic category, Hausdorff dimension and dimension coincide. So the real dimension of ${\cal D}$ is $\leqslant 2$.

Suppose that $x_0$ is a point in ${\cal D}^\prime$. 
Let $z$ be the unique double point of $f_{\restriction
\pi^{-1}(x_0)}$. As described in the proof of
Proposition \ref{algebraicProposition}, we have
a chart $\psi:\R^4 \times \C^2 \rightarrow \CP^3$ given by
$\psi(\vec{x},\lambda) \rightarrow [p_1(\vec{x}) + \lambda
p_2(\vec{x}) ]$. We assume this chart is centred on $z$.
Suppose that $\Sigma$ is defined by the
polynomial $f$ and consider the function $F(\vec{x},\lambda)=f
\circ \psi$. $F$ is holomorphic in the
$\lambda$ component and smooth in the $\vec{x}$ component.
Since $z$ lies in ${\cal D}^\prime$ we have that $F$ and
$\frac{\partial F}{\partial \lambda}$ vanish at $0$ but that
$\frac{\partial^2 F}{\partial \lambda^2}$ does not.
Define the map $g:\R^4 \times \C \rightarrow \C$ by
$g=\frac{\partial F}{\partial \lambda}$. $\Sigma$ is tangent
to the fibre at $z$. $g_*( \frac{\partial}{\partial
\lambda})\neq 0$. Therefore when restricted to $T_z\Sigma$,
the differential $g_*$ has real rank $2$. Note that the fact $F$
is holomorphic in $\lambda$ is the crucial point here.
By the implicit function theorem $g^{-1}(0) \cap \Sigma$ is
locally a smooth $2$-manifold in a neighbourhood of $z$. Since
$x_0 \in {\cal D}^\prime$,
$\pi$ is a homeomorphism of a neighbourhood of $z$ onto a
neighbourhood of $x_0$.
Thus ${\cal D}^\prime$ is a smooth 2-manifold.

At $z$ the tangent space $T_z\Sigma$ can be written as the sum
of a vertical subspace $V$, given by the tangent space of the
fibre and a horizontal subspace $H$ with $\pi_*$ mapping $H$ isomorphically
to the tangent space of the discriminant locus.
There is a canonical symplectic form $\omega_V$ on V given by
the pull-back of the Fubini-Study metric onto the fibre.
Similarly there is a canonical symplectic form $\omega_\Sigma$ on
$T_z \Sigma$. So we can say that a non-degenerate two form
$\omega_H$ is positively oriented if $\omega_V \wedge \omega_H$
is a positive multiple of $\omega_\Sigma$. This defines an
orientation of $H$ and hence an orientation $T_{x_0}{\cal
D}^\prime$.
\end{proof}

It is quite possible that the dimension of the discriminant locus
of a surface $\Sigma$ is strictly less than $2$. For example:
the discriminant locus of a plane is just a point; there exist
quadrics whose discriminant locus is just a circle (\cite{salamonViaclovsky})). The discriminant locus of a reducible surface is given by the union of the discriminant loci of the intersections and the image of the intersection of components under $\pi$. This allows one to manufacture singular surfaces with pathological discriminant loci. For example consider the cubic surface consisting of three planes intersecting in a line. If the line of intersection is a twistor line, the discriminant locus is a single point. If the line of intersection is not a twistor line, then the discriminant locus is a round sphere and all points in the  discriminant locus are triple points or twistor lines.

\begin{proposition}
If $\Sigma$ is a non-singular complex surface of degree $d$ with $d>2$ then its discriminant locus has dimension $2$.
\end{proposition}
\begin{proof}
We have already shown that the dimension of the discriminant locus is no greater than $2$, so suppose for a contradiction that the discriminant locus ${\cal D}$ is $1$ (or less) dimensional. Then $S^4 \setminus {\cal D}$ is simply connected. The restriction $\pi_{\restriction \Sigma \setminus (\pi^{-1}{\cal D})}$ is a $d$-sheeted cover of $S^4 \setminus {\cal D}$ with no branch points. Therefore $\Sigma \setminus (\pi^{-1}{\cal D})$ has $d$ connected components. In particular it is disconnected if $d \geq 2$.

If $\Sigma$ is non-singular and $d>2$ then $\Sigma$ only contains a finite number of projective lines (see \cite{sarti} for a survey of the known bounds). Given a point $x \in S^4$ then $\pi^{-1}(x)$ is either finite or a projective line. So $\pi^{-1}{\cal D}$ has dimension at most $2$. Non-singular complex surfaces are always connected. So $\Sigma \setminus (\pi^{-1}{\cal D})$ is connected.

This is the desired contradiction.
\end{proof}

We next compute the Euler characteristic:

\begin{proposition}
Let $\Sigma$ be a non-singular cubic curve. Let $T$ denote the set
of triple points of the twistor fibration and ${\cal D}$
be the discriminant locus. We have:
\[ \chi({\cal D}) = -3 - \chi(T) \]
\label{eulerCharacteristic}
\end{proposition}
\begin{proof}
Let $n$ denote the number of twistor lines. $\pi_{\restriction \Sigma}$
is a 3 sheeted cover of $S^4$ branched over the discriminant locus. There are two points in $\Sigma$ over each point in ${\cal D}^\prime$, $1$ point over each triple point and a copy of $\CP^1$ for each twistor line. Hence by the additivity of the Euler characteristic:
\[ \chi(\Sigma) =3 \chi(S^4 \setminus {\cal D}) + 2\chi({\cal
D}^{\prime}) + \chi(T) + n \chi(\CP^1) \]
All non-singular cubics are diffeomorphic to $\CP^2 \#
6\ol{\CP^2}$. So $\chi(\Sigma)=9$.
\[ 9    =  6 - 3\chi({\cal D}) + 2(\chi({\cal D}) - \chi(T) -n) + \chi(T) + 2 n \]
The result follows.
\end{proof}

\begin{corollary}
The set ${\cal D}\setminus{\cal D}^\prime$ on a smooth non-singular surface $\Sigma$ of degree $d$ is always non-empty if $d$ is odd.
\end{corollary}
\begin{proof}
Suppose that $d$ is odd and that the twistor fibration contains no triple points or
twistor lines. The discriminant locus will then be smooth, compact and orientable. So its Euler characteristic must be even. 

The Euler characteristic of $S^4$ is even. Hence the Euler characteristic of a $d$ fold cover of $S^4$ branched along the discriminant locus must also be even. In particular the Euler characteristic of $\Sigma$ must be even.

On the other hand, the Euler characteristic of a degree $d$ complex surface is $d(6-d(4-d))$ (This can be proved using the adjunction formula and the fact that the Euler characteristic is equal to the top Chern number: see \cite{griffithsAndHarris}).
\end{proof}

In this section we have proved a number of basic topological facts about the discriminant locus, but much remains unanswered. For example can we obtain bounds on the number of connected components of ${\cal D}$ or on the number of triple points? One can find crude
bounds rather easily using the Milnor--Thom bound (\cite{milnor, thom}). This states that the sum of the Betti numbers of a real algebraic variety in $\R^m$ defined as the zero set of a finite number of polynomial equations of degree $k$ is less than or equal to $k(2k-1)^{m-1}$. Thus for a cubic surface the sum of the Betti numbers of ${\cal D}$ is bounded above by $12(23)^3=146004$ and the number of triple points is bounded above by $8(15)^3=27000$. One can surely do rather better! 

Note that for the related question of the number of twistor lines, we have a sharp bound: there are at most $5$ twistor lines. See \cite{armstrongSalamon} for a proof and a classification of cubic surfaces with $5$ twistor lines.

\section{The topology of singular surfaces}
\label{singularSurfacesSection}

The primary aim of this paper is to compute the topology of the discriminant locus of the Fermat cubic. But what does it actually mean to compute the topology of a space $X$? The very notion presupposes that one has a classification theorem for the candidate topologies and one wishes to identify the topology of $X$ within this classification. In order to compute the topology of the discriminant locus, therefore, we need some classification for the topology of singular surfaces. Since we expect that for generic cubics, the discriminant locus will have isolated singularities, let us focus on this case.

\begin{definition}
A topological space is a {\em real surface with isolated topological singularities} if it is homeomorphic to a finite CW-complex built using only $0$, $1$ and $2$ cells where precisely two faces meet at any given edge. This ensures that along the interior of the edges the space is locally homeomorphic to $\R^2$.

A {\em topological singularity} is a point on such a surface which is not locally homeomorphic to $\R^2$. We will sometimes abbreviate this to simply {\em singularity} when there is no danger of confusion with the algebraic notion of singularity.  We empahsise the distinction because we will find in practice that the triple points of the Fermat cubic are algebraically singular but not topologically singular.
\end{definition}

The topological classification of such spaces is straightforward but perhaps not very well known. We will show how to classify theses spaces by means of an associated graph which encodes the topology.  See \cite{fortuna2005effective} for an alternative, but equivalent, description for the case of surfaces in ${\mathbb{RP}}^3$.

Let us begin by describing the associated graph in some special cases, we will then explain more formally how to construct the graph.

For a non-singular compact real surface, the graph consists of a
single node labelled either $\Sigma_g$ or $\Theta_c$ with $g \in \N$ and $c \in \N^+$. $\Sigma_g$ is the label used for a connected orientable surface of genus $g$. $\Theta_n$ is the label used for a connected non-orientable surface with $c$ cross caps.

When $n$ smooth points of a singular real surface are glued
together at a point we add one extra node to the graph
representing the glue point. 
We join this node to the nodes representing the smooth parts of
the surface that have been glued
together.

\begin{example} The surface obtained by choosing two points on a
sphere, gluing one point to a torus and the other
to a Klein bottle has graph:

$\Sigma_1$---$\circ$---$\Sigma_0$---$\circ$---$\Theta_2$.
\end{example}
\begin{example}
If one takes three points on a sphere and glues them all
together, the resulting surface has graph
\chemfig{{\Sigma_0}~\circ}.
\end{example}

Let us now describe in general how to associate a graph
$\Gamma_\Delta$ to a cell decomposition $\Delta$. The point we wish
to emphasize is the algorithmic way in which one can
compute $\Gamma_\Delta$ from $\Delta$.

First, we need to understand the topology of the surface at the
singularities. Given a vertex $v$ of the cell decomposition we can
define a graph $\Gamma_v$ as follows: add a node to the graph
for each edge that ends at $v$;
for each face in the cell decomposition which has two edges that meet
at $v$ in its boundary,
connect the corresponding nodes in the $\Gamma_v$.
Since two faces meet at each edge, $\Gamma_v$ will consist
simply of a closed loops. Let $n_v$ denote the number
of loops. If $n_v\neq 0$, the surface is locally homeomorphic to
$n_v$ discs whose origins have all
been glued together with $v$ corresponding to the glue point. In
the case $n_v=0$ the surface consists of a single point.
Note that in the case $n_v=1$, the surface is locally
homeomorphic to $\R^2$.

We can now define the resolution of a singularity to be the
cell decomposition obtained by replacing the vertex with $n_v$
vertices each connected to one of the components of $\Gamma_v$.
This construction simply corresponds to ungluing the discs.

Given a cell decomposition $\Delta$ we now define the nodes of the
graph $\Gamma_\Delta$. The nodes of $\Gamma_\Delta$ are defined to be
the union of two sets $N_1$ and $N_2$. The set of nodes $N_1$ 
is given by the set of connected components of the topological space
$\Delta$ with its vertices removed. By resolving
the singularities of each connected component, we obtain
a compact smooth real surface. We add the data of its Euler
characteristic and orientation to each node in $N_1$.

The nodes $N_2$ for $\Gamma_\Delta$ are defined to be
the topological singularities. That is they correspond
to the vertices of $\Delta$ with $n_v \neq1$. There is no
additional data associated with these nodes.

The links in $\Gamma_\Delta$ are defined as follows: we add one link to the node $v \in N_2$ for each connected
component of $\Gamma_v$; this link connects the node $v$
to the node in $N_1$ containing the edges and faces of
associated with the connected component. There are no
other links.

With the obvious notion of graph homomorphism for this category of graphs, the graph $\Gamma_\Delta$ completely classifies the surface up to homeomorphism. This follows from the classification theorem for closed non-singular surfaces.

We can now state the aim of this paper more clearly. It
is to compute the graph that encodes the topology of the discriminant locus of the Fermat cubic.

\section{The discriminant locus of the Fermat cubic}
\label{fermatSection}

The discriminant locus defines a surface in $\R^4 \cup \{\infty\}$ which we can view as the world sheet swept out as a curve moves through $\R^3$. A first step to understanding the discriminant locus
is to generate an animation of this curve. One hopes that by careful study of the resulting curve we should be able to piece together
the topology of the surface. We will do this in two stages --- firstly we will calculate the topology by a simple visual examination of the curve. We will then show how the results of this visual examination can be rigorously justified.

The ease with which one can comprehend the animation of the discriminant locus depends crucially upon the choice of coordinates for $\R^4$. The ideal coordinates for viewing the Fermat cubic are
not immediately apparent. As a preliminary step to choosing good
coordinates, let us describe the conformal symmetries of the Fermat cubic. This will surely guide our choice of coordinates.

The Fermat cubic is defined by the equation
\[ z_1^3 + z_2^3 + z_3 ^3 + z_4^3 = 0.\]
There is an obvious $S^4$ action given by permuting the
coordinates. One can
also make transformations such as $(z_1,z_2,z_3,z_4) \mapsto
(\omega_1 z_1, \omega_2 z_2, \omega_3 z_3, \omega_4 z_4)$ where
each $\omega_i$ is a cube root of unity. These generate all the
projective symmetries of the Fermat cubic.

It is easy now to check that the conformal symmetries of the Fermat cubic are generated by $(z_1,z_2,z_3,z_4) \mapsto (z_3,z_4,z_1,z_2)$
and $(z_1,z_2,z_3,z_4) \mapsto (\omega z_1, \omega^2 z_2,
z_3, z_4)$ where $\omega=e^{\frac{2\pi i}{3}}$.
Thus the group of conformal symmetries is isomorphic to $\Z_3 \times \Z_2$.

The next observation to make is that the Fermat cubic has three
twistor lines. This is easily checked using the well known explicit description for the $27$ lines on the Fermat cubic.

If one then chooses coordinates for $S^4$ such that the twistor lines are at $(-1,0,0,0)$, $(1,0,0,0)$ and $\infty$ then by Proposition
\ref{algebraicProposition} the equations for the discriminant locus are equations of degree $8$ and the equations for the triple points are of degree $4$. The equations for the triple points are sufficiently simple for Mathematica to be able to solve them. This demonstrates the practical value of Proposition \ref{algebraicProposition}.

If we then transform back to the standard coordinates we have:
\begin{lemma}
Identifying $S^4$ with $\H \cup \{ \infty \}$, the triple points
for the Fermat cubic have coordinates
$-\frac{1}{2} j + \frac{\sqrt{3}}{2}k$,
$\frac{1}{2}j + \frac{\sqrt{3}}{2}k$,
$-j$,
$j$,
$-\frac{1}{2}j -\frac{\sqrt{3}}{2}k$ and
$\frac{1}{2}j, -\frac{\sqrt{3}}{2}k.$

The twistor lines for the Fermat cubic lie above the points
$-1$, $\frac{1}{2} + \frac{\sqrt{3}}{2}i$, $\frac{1}{2}
-\frac{\sqrt{3}}{2}i$.
\label{triplePointLemma}
\end{lemma}

Notice that all of these points lie on the unit $3$-sphere. This
tells us that we can make a quaternionic M\"obius transformation to transform these points so they all lie in the plane $x_1=0$. Then, if we produce an animation with time coordinate $x_1$ we will be able to view all these singular points simultaneously at time
$x_1=0$.

One such M\"obius transformation is $q \mapsto
(q-k)^{-1}(-q-k)$. It transforms the six triple points to the
following points:
$-(2 + \sqrt3) i$, $(2 + \sqrt3) i$, $-i$, $i$, $-(2 -
\sqrt3)i$, $(2-\sqrt3)i$. Note that they all lie on the $j$
axis. The points corresponding to the twistor lines are mapped
to $-\frac{\sqrt3}{2}j - \frac{1}{2}k$, $+\frac{\sqrt3}{2}j -
\frac{1}{2}k$
and $k$. They lie in a circle in the $j$, $k$ plane.

\begin{figure}[htbp]
\begin{tabular}{ll}
The discriminant locus at time $x^\prime_1=0$: & \\
\includegraphics[scale=0.45]{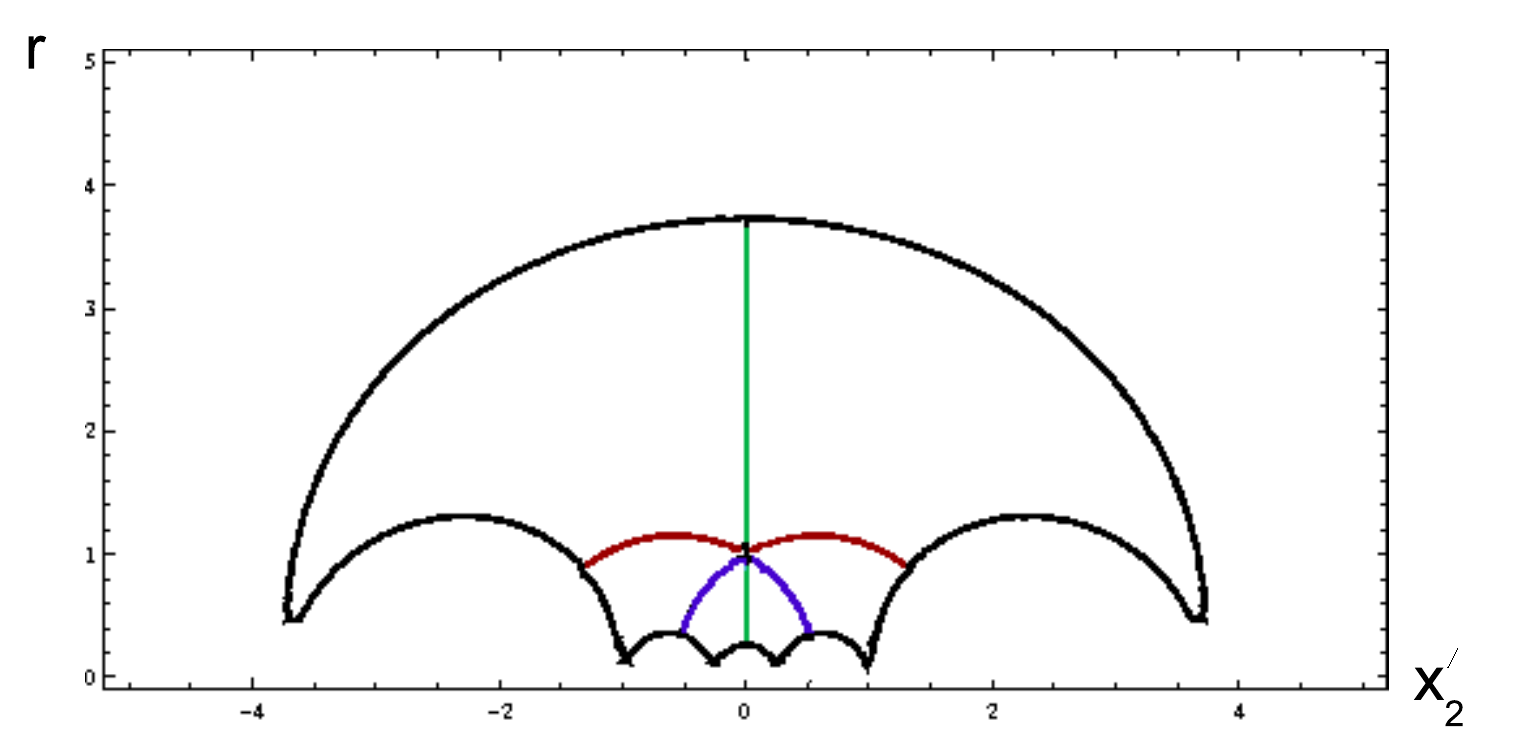} & \\
\includegraphics[scale=0.45]{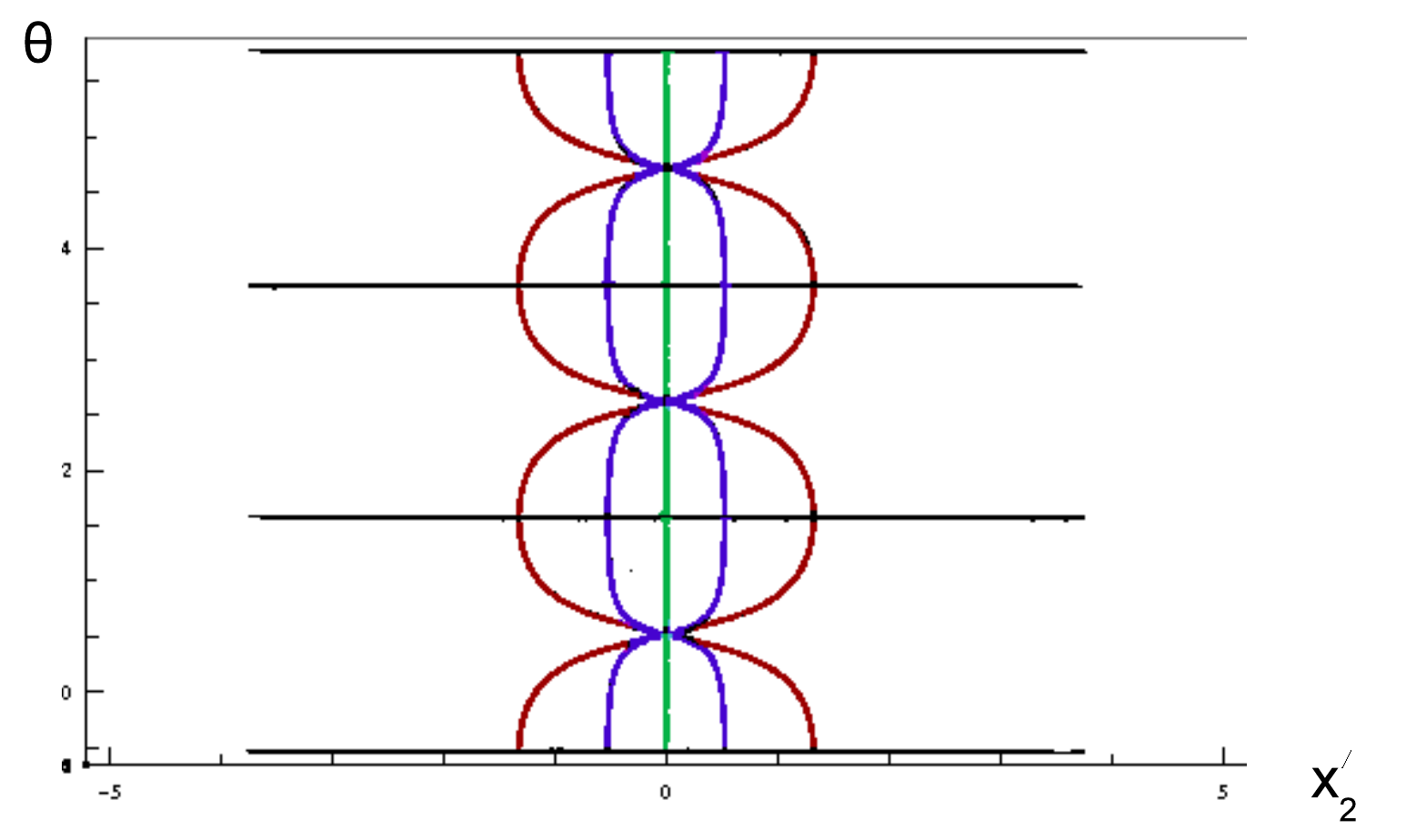} 
& \includegraphics[scale=0.45]{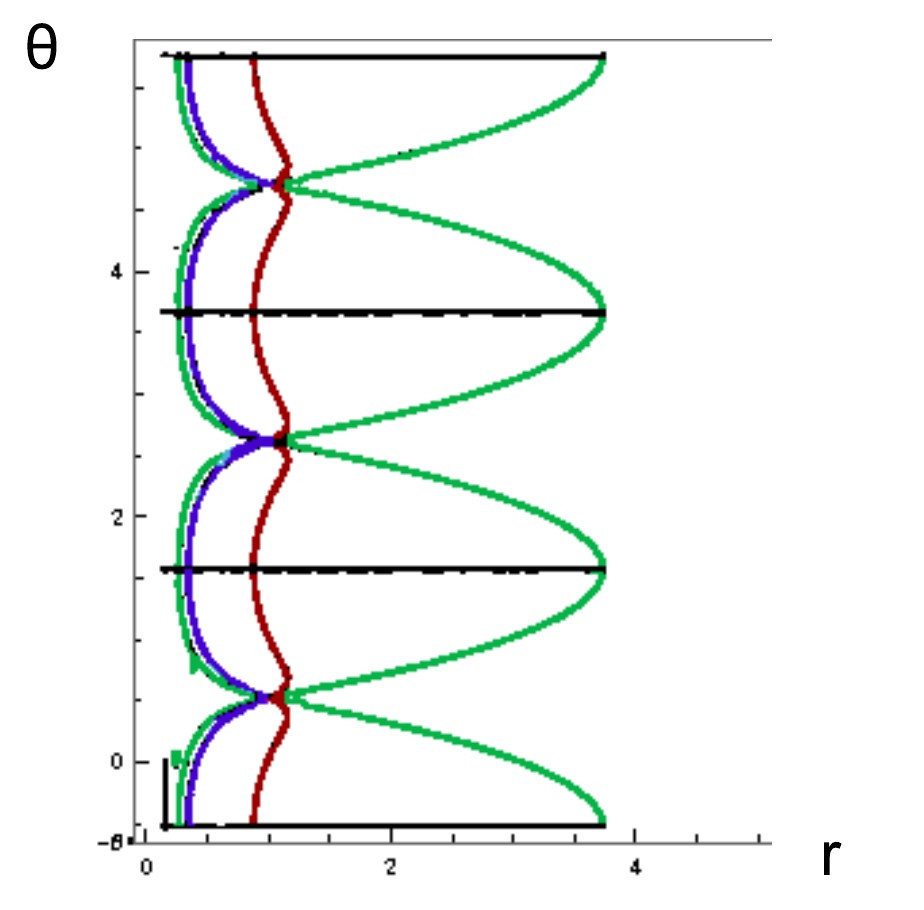} \\
The discriminant locus at time $x^\prime_1=0.02$: & \\
\includegraphics[scale=0.45]{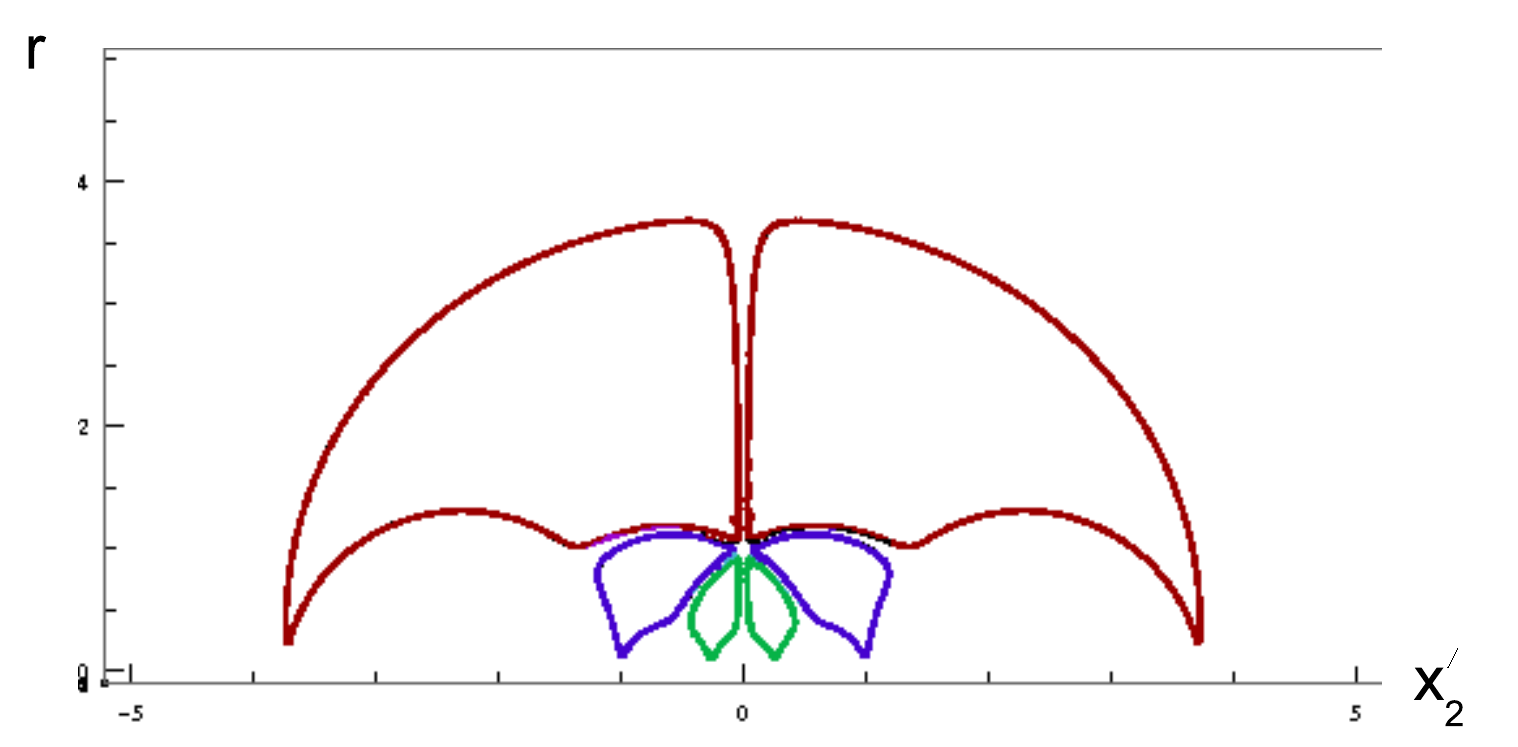} & \\
\includegraphics[scale=0.45]{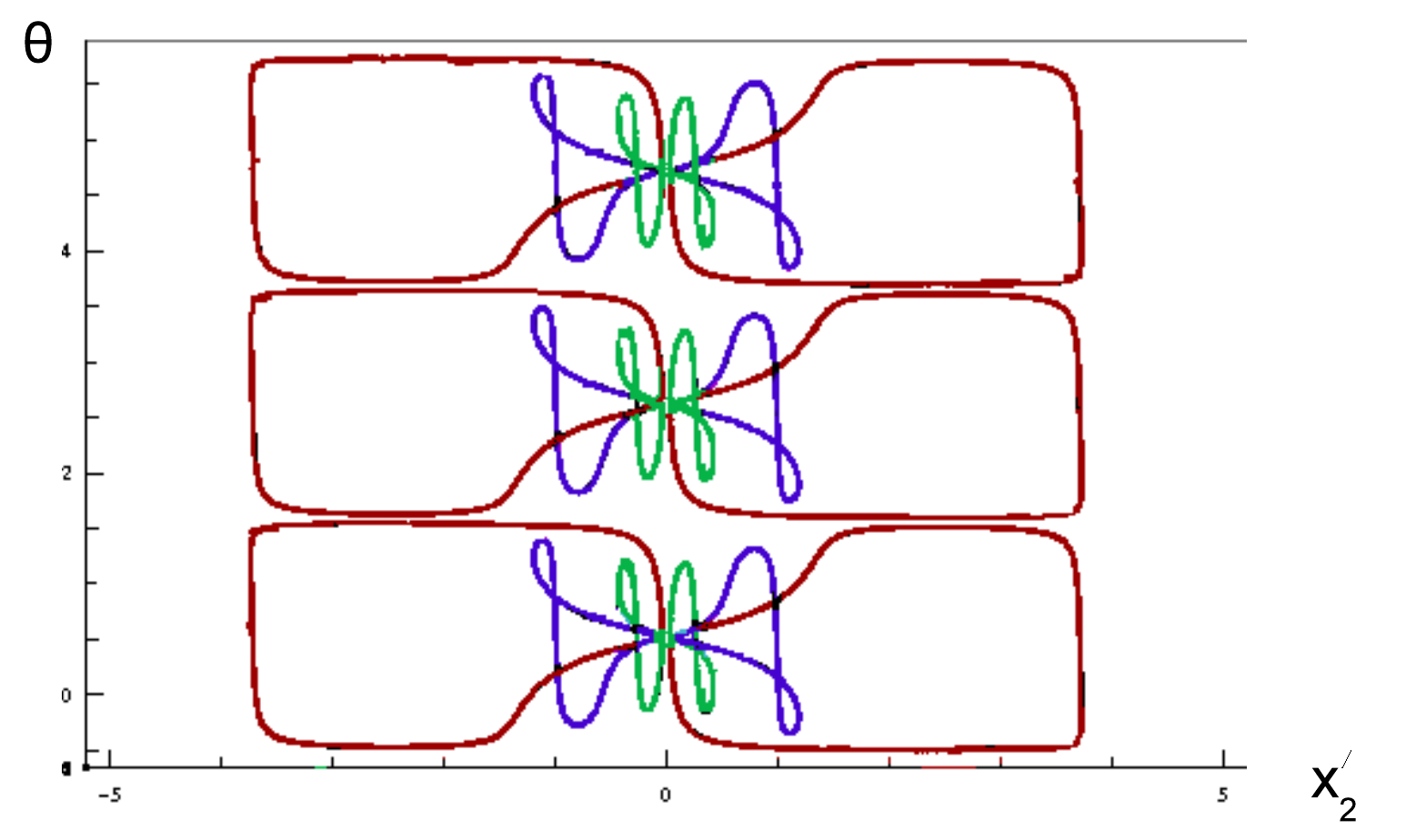} &
\includegraphics[scale=0.45]{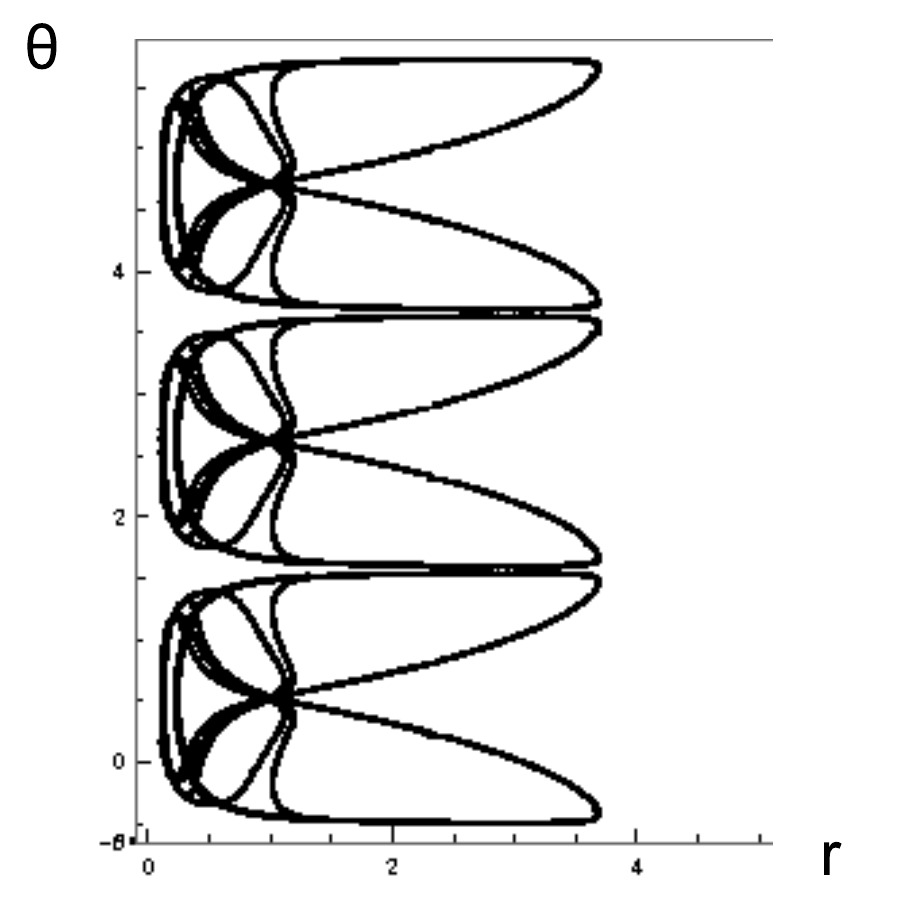} \\
\end{tabular}
\caption{Three views of the discriminant locus at two different
times}
\label{fermatDiscriminantLocus}
\end{figure}

\begin{figure}[htbp]
\label{fermatSchematic}
\includegraphics[scale=1]{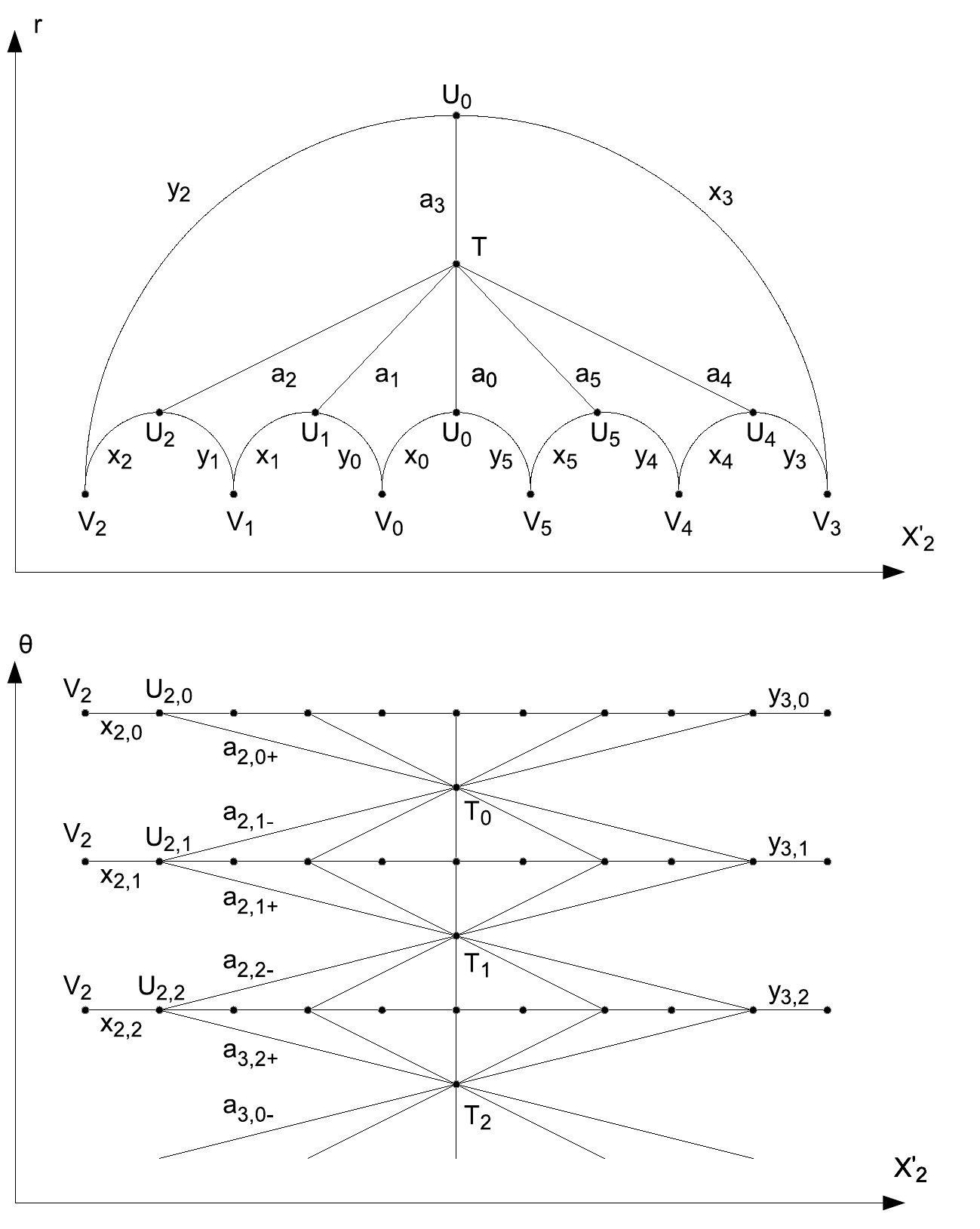}
\caption{Labels for the vertices and edges of the cell decomposition}
\end{figure}

Let us use the notation $\vec{x^\prime}$ for these transformed
coordinates. In these coordinates, the animation of the discriminant locus becomes considerably simpler.

One feature that stands out when using these coordinates is an apparent rotation symmetry given by rotating through $120$ degrees about the $i$ axis. This suggests that it might be useful to introduce cylindrical polar coordinates. We write $x^\prime_3=r \cos \theta$ and $x^\prime_4=r \sin \theta$. 

In Figure \ref{fermatDiscriminantLocus} we have used these
cylindrical polar coordinates to plot the discriminant locus at each of the times $x^\prime_1=0$ and $x^\prime_1=0.02$. We
have arranged the pictures so that the reader can (mentally or
physically) fold the page to obtain a three dimensional view. We have omitted the points where $r=0$ since they are a little confusing since points with $r=0$ need to be identified appropriately in polar coordinates. If we had included the points there would be six vertical lines representing the triple points running across the front of the figure at time $x^\prime_1=0$.

For each time point we have coloured the plots to indicate which
parts of the different views correspond. We have not coloured the view of the $(r, \theta)$ plane at time $0.02$ as the non-transverse intersections make it difficult to work out how the curves correspond to the other figures. Note that there is no relationship between the colours used in the colours used at times $x^\prime_1=0$ and $x^\prime_1=0.02$.

The first thing to notice about Figure \ref{fermatDiscriminantLocus} is the translation symmetry on the $\theta$ axis (corresponding to 120 degree rotations in the $x^\prime$ coordinates). If one simplifies the expression for the discriminant locus in these coordinates one obtains an expression only involving $\theta$ via the functions $\cos(3 \theta)$ and $\sin(3 \theta)$. This proves the visually apparent symmetry is a genuine symmetry of the discriminant locus.

This is surprising since this symmetry {\em does not} arise from the
group of conformal symmetries $\Z_3 \times \Z_2$ acting on the
Fermat cubic. We conclude that the discriminant locus is more symmetrical than the Fermat cubic itself. We will discover as we continue our study that it has yet more symmetries.

Next note that at small positive $x^\prime_1>0$ times the discriminant locus consists of a number of disjoint loops. We have illustrated this behaviour with only one time point $x^\prime_1=0.02$, but it appears to be true for all positive times when one plots an animation. As time progresses, each of these loops shrinks to a point and then disappears. There is a time symmetry in these coordinates which ensures that at small negative times, one similarly sees a number of disjoint loops which shrink to a point and then disappear as time $x^\prime_1$ decreases.

Since the worldsheet of a loop shrinking to a point and disappearing is homeomorphic to $\R^2$, this suggests that we have found a cell decomposition of the discriminant locus.

The vertices and edges of the cell decomposition are given by the curve at time $x^\prime_1=0$. Plotting the curve at this time does indeed show it to be singular: the singularities are the vertices of our cell decomposition, the smooth curves are the edges. The worldsheets of the loops at positive and negative times give the faces of our cell decomposition.

Note that the view of the $(x^\prime_2,r)$ plane clearly shows $6$
distinct loops at time $t=0.02$. As time progresses these
loops shrink to points and disappear. Taking into account the
$\Z_3$ symmetry, this gives a total of $18$ loops for
positive times. There is also a time symmetry given by
$(x_1^\prime,x_2^\prime, r, \theta) \mapsto
(-x_1^\prime,-x_2^\prime,r,\theta)$. Thus we have identified a
cell decomposition with $36$ faces.

So, by means of visual inspection, we have identified a cell decomposition of the discriminant locus. We have two tasks remaining: the first is to compute the topology of the discriminant locus from the cell decomposition; the second is to rigorously justify the existence of the cell decomposition. The first task is reasonably simple. We postpone the second to the next section.

\begin{proposition}
\label{topologyCalculation}
Assuming the cell decomposition of the discriminant locus of the Fermat cubic obtained by visual inspection is correct, its topology is encoded by the graph:
\[
\begin{tikzpicture}[scale=1]
\node (n) at (1,1) {$\Sigma_4$} ;
\node [circle, draw, fill=white,inner sep=0pt, minimum width=4pt] (a) at (0,0) {};
\node [circle, draw, fill=white,inner sep=0pt, minimum width=4pt] (b) at (1,0) {};
\node [circle, draw, fill=white,inner sep=0pt, minimum width=4pt] (c) at (2,0) {};
\draw [double] (a) -- (n); 
\draw [double] (b) -- (n); 
\draw [double] (c) -- (n); 
\end{tikzpicture}
\]
\end{proposition}
\begin{proof}
We label the edges and vertices of the cell decomposition as shown in Figure \ref{fermatSchematic}.

In the top part of Figure \ref{fermatSchematic} we have shown how to associate labels for the edges and vertices seen in the view of the $(x_2^\prime,r)$ plane at time $x_1^\prime=0$. We have used a more schematic representation of the decomposition than that shown in Figure \ref{fermatDiscriminantLocus}, but the relationship between the pictures should be obvious. We have used lower case letters to indicate edges. We have used upper case letters to indicate vertices. We have used indices from $0$ to $5$ to correspond to the approximate rotational symmetry about the point $(0,1)$ in the $(x_2^\prime,r)$ plane. Our numbering starts at the downward pointing vertical and then moves clockwise.

Since the $(x_2^\prime,r)$ plane is only two dimensional, some of the edges and vertices of our cell decomposition coincide when viewed in the $(x_2^\prime,r)$ plane. We have shown in the lower part of Figure \ref{fermatSchematic} how to add a second index to the edges and vertices to indicate their $\theta$ coordinate. In summary then our edges and vertices can be enumerated as follows:
\begin{itemize}
\item Edges $a_{i,j} \quad i \in \{0,1,2,3,4,5\}, \quad j \in \{0+,1-,1+,2-,2+,0-\}$;
\item Edges $x_{i,j} \quad i \in \{0,1,2,3,4,5\}, \quad j \in \{0,1,2\}$;
\item Edges $y_{i,j} \quad i \in \{0,1,2,3,4,5\}, \quad j \in \{0,1,2\}$;
\item Vertices $U_{i,j} \quad i \in \{0,1,2,3,4,5\}, \quad j \in \{0,1,2\}$;
\item Vertices $V_i \quad i \in \{0,1,2,3,4,5\}$;
\item Vertices $T_j \quad j \in \{0,1,2\}$.
\end{itemize}
Notice that we only have $6$ vertices of the form $V_i$ since when $r=0$ we must identify points that only differ by the value of $\theta$.

We label the faces of our cell decomposition $f^\pm_{i,j}$ with $i \in \{0, 1, 2, 3, 4, 5\}$ and $j \in \{0, 1, 2\}$. Here the plus or minus indicates whether the face corresponds to positive times $x_1^\prime > 0$ or to negative times. The $i$ index indicates the sector of the $(x_2^\prime,r)$ plane and the $j$ indicates the values of $\theta$ in accordance with the conventions used for edges and vertices.

An examination of Figure \ref{fermatDiscriminantLocus} allows us to write down the boundary of each face:
\begin{eqnarray*}
\partial (f^+_{i,j}) & = & x_{i,j+1} + a_{i,(j+1)-}+a_{i+1,j+} + y_{i,j} \\
\partial (f^-_{i,j}) & = & x_{i,j} + a_{i,j+} + a_{i+1,(j+1)-} + y_{i,j+1}
\end{eqnarray*}
In the formulae above, arithmetic in indices is performed modulo $6$ when working with the $i$ indices and modulo $3$ when working with the $j$ indices. We will use modular arithmetic for $i$ and $j$ indices for the rest of the proof without further comment.

One can see immediately from the above formulae for the faces in our cell decomposition that precisely two faces meet at each edge. Thus the cell decomposition describes a surface with isolated topological singularities.

We compute the local topology at a vertex $v$ by drawing the graph $\Gamma_v$ as described in Section \ref{topologicalSection}. The graph $\Gamma_{U_{i,j}}$ is:
\[
\begin{tikzpicture}[scale=1.5]
\node (i) at (1,0) {$x_{i,j}$} ;
\node (j) at (0,1) {$a_{i,j-}$};
\node (k) at (-1,0) {$y_{i-1,j}$};
\node (l) at (0,-1) {$a_{i,j+}$};
\draw [-] (i) -- node[right] {$f_{i,j-1}^+$} (j);
\draw [-] (j) -- node[left] {$f_{i-1,j-1}^-$}(k);
\draw [-] (k) -- node[left] {$f_{i-1,j}^+$}(l);
\draw [-] (l) -- node[right] {$f_{i,j}^-$}(i);
\end{tikzpicture}
\]
The graph $\Gamma_{V_i}$ is:
\[
\begin{tikzpicture}[scale=1]
\node (a) at (4,1.73205080757) {$x_{i,0}$} ;
\node (b) at (3,3.46410161514) {$y_{i,1}$};
\node (c) at (1,3.46410161514) {$x_{i,2}$};
\node (d) at (0,1.73205080757) {$y_{i,0}$};
\node (e) at (1,0) {$x_{i,1}$};
\node (f) at (3,0) {$y_{i,2}$};
\draw [-] (a) -- node[right] {$f_{i,0}^-$} (b);
\draw [-] (b) -- node[above] {$f_{i,1}^+$} (c);
\draw [-] (c) -- node[left] {$f_{i,2}^-$} (d);
\draw [-] (d) -- node[left] {$f_{i,0}^+$} (e);
\draw [-] (e) -- node[below] {$f_{i,1}^-$} (f);
\draw [-] (f) -- node[right] {$f_{i,2}^+$} (a);
\end{tikzpicture}
\]
The graph $\Gamma_{T_{j}}$ is:
\[
\begin{tikzpicture}[scale=1]
\node (a1) at (4,1.73205080757) {$a_{0,j+}$} ;
\node (b1) at (3,3.46410161514) {$a_{1,(j+1)-}$};
\node (c1) at (1,3.46410161514) {$a_{2,j+}$};
\node (d1) at (0,1.73205080757) {$a_{3,(j+1)-}$};
\node (e1) at (1,0) {$a_{4,j+}$};
\node (f1) at (3,0) {$a_{5,(j+1)-}$};
\draw [-] (a1) -- node[right] {$f_{0,j}^-$} (b1);
\draw [-] (b1) -- node[above] {$f_{1,j}^+$} (c1);
\draw [-] (c1) -- node[left] {$f_{2,j}^-$} (d1);
\draw [-] (d1) -- node[left] {$f_{3,j}^+$} (e1);
\draw [-] (e1) -- node[below] {$f_{4,j}^-$} (f1);
\draw [-] (f1) -- node[right] {$f_{5,j}^+$} (a1);
\node (a2) at (10,1.73205080757) {$a_{0,(j+1)-}$} ;
\node (b2) at (9,3.46410161514) {$a_{1,j+}$};
\node (c2) at (7,3.46410161514) {$a_{2,(j+1)-}$};
\node (d2) at (6,1.73205080757) {$a_{3,j+}$};
\node (e2) at (7,0) {$a_{4,(j+1)-}$};
\node (f2) at (9,0) {$a_{5,j+}$};
\draw [-] (a2) -- node[right] {$f_{0,j}^+$} (b2);
\draw [-] (b2) -- node[above] {$f_{1,j}^-$} (c2);
\draw [-] (c2) -- node[left] {$f_{2,j}^+$} (d2);
\draw [-] (d2) -- node[left] {$f_{3,j}^-$} (e2);
\draw [-] (e2) -- node[below] {$f_{4,j}^+$} (f2);
\draw [-] (f2) -- node[right] {$f_{5,j}^-$} (a2);
\end{tikzpicture}
\]

We deduce that the discriminant locus has precisely three  topological singularities $T_0$, $T_1$ and $T_2$ corresponding to the twistor lines. At these singularities, the discriminant locus is locally homeomorphic to two discs glued together at a point.

Since the faces $f_{i,j}^+$ and $f_{i,j}^-$ lie in a connected component component of the graph $\Gamma_{V_i}$, these faces can be connected via a path that avoids all vertices. The graph $\Gamma_{V_i}$ also shows that the faces $f_{i,j_1}^\pm$ and $f_{i,j_2}^\pm$ can be connected by such a path for all $j_1$ and $j_2$. Similarly the graph of $\Gamma_{T_j}$ shows that the faces $f_{i_1,j}^\pm$ and $f_{i_2,j}^\pm$ can be connected by such a path for all $i_1$ and $i_2$. We deduce that all faces can be connected by a path that avoids the singularities.

We know from Proposition \ref{orientabilityProposition} that discriminant locus is orientable away from the triple points and twistor lines. We deduce from this and the considerations above that the topology of the discriminant locus is described by a graph of the form:
\[
\begin{tikzpicture}[scale=1]
\node (n) at (1,1) {$\Sigma_g$} ;
\node [circle, draw, fill=white,inner sep=0pt, minimum width=4pt] (a) at (0,0) {};
\node [circle, draw, fill=white,inner sep=0pt, minimum width=4pt] (b) at (1,0) {};
\node [circle, draw, fill=white,inner sep=0pt, minimum width=4pt] (c) at (2,0) {};
\draw [double] (a) -- (n); 
\draw [double] (b) -- (n); 
\draw [double] (c) -- (n); 
\end{tikzpicture}
\]
Our cell decomposition contains $27$ vertices, $72$ edges and $36$ faces
so has Euler characteristic $-9$. One can alternatively calculate this using Lemma \ref{triplePointLemma} and Proposition \ref{eulerCharacteristic}. The Euler characteristic of the singular surface associated with the graph above is $(2-2g)-3$. So $g=4$.
\end{proof}

\section{Symmetry of the discriminant locus}
\label{symmetrySection}

The first step towards a rigorous justification of the above results is to find a simpler expression for the discriminant locus. To do this we need to exploit all of its symmetries.

Assuming that our topological calculation above is correct, one sees that the image of the twistor lines on the Fermat cubic's discriminant locus can be identified entirely in terms of the geometry of the discriminant locus without reference to the Fermat cubic. They are simply the points of the discriminant locus that are not locally homeomorphic to $\R^2$. Thus any conformal symmetry of the discriminant locus must preserve the $3$ points corresponding to the twistor lines.

Similarly one can identify the triple points as the points of the discriminant locus that are not smoothly embedded in $S^4$ yet are locally homeomorphic to $\R^2$. Thus any conformal symmetry of the discriminant locus must preserve the $6$ triple points. We can use this to put a bound on the size of the group of conformal symmetries of the discriminant locus of the Fermat cubic.

\begin{lemma}The group, $G$, of conformal symmetries that permute the six triple points of the Fermat cubic and its three twistor lines is isomorphic to $D_3 \times D_6$.
\end{lemma}
\begin{proof}
By Lemma \ref{triplePointLemma}, we have identified a round $3$-sphere in $S^4$ that contains all $9$ non-smooth points. It is the unique such sphere, so it too must be preserved by any element of $G$. There is a one to one correspondence between conformal symmetries of $\R^4$ that leave this $3$-sphere fixed and angle preserving (but not necessarily orientation preserving) symmetries of the $3$-sphere.

The six triple points all lie on a hexagon in the unit circle in the plane spanned by the quaternions $j$ and $k$ and the three twistor lines lie on an equilateral triangle in the unit circle in the plane spanned by $1$ and and $i$. The symmetry group of a regular $n$-gon in the plane is $D_n$. Thus we can find a subgroup $D_3 \times D_6$ of $\subset O(2) \times O(2) \subset O(4)$ acting on $\R^4$ which preserves the twistor lines and triple points. The induced maps on the unit sphere $S^3$ are isometries and hence angle preserving. Thus $G$ contains $D_3 \times D_6$.

To show that there are no further symmetries, we switch to the ${\vec x}^\prime$ coordinates. The unit $3$-sphere becomes the hyperplane $x^\prime_1=0$ together with the point at infinity. Let $C_3$ be the unit circle in the $x^\prime_1=x^\prime_2=0$ plane containing the image of the twistor lines. Let $C_6$ be the line $x^\prime_1=x^\prime_3=x^\prime_4=0$ containing all $6$ triple points.

We can make a further conformal change of coordinates on the three sphere that fixes $C_3$ and $C_6$ but which moves one of the triple points to infinity and another to zero. In these coordinates, any quaternionic M\"obius transformation that fixes the triple points must be given by a linear map. This map must be conformally equivalent to a unitary map. If this map also fixes $C_3$ it must be an isometry. Thus the group of transformations that fix $C_6$ and permute $C_3$ is $D_3$.

We can make a conformal transformation of the three spheres that swaps $C_3$ and $C_6$. The argument above then shows that the group of transformations that fix $C_3$ and permute $C_6$ is the dihedral group $D_6$. The result follows.
\end{proof}
\begin{lemma}The discriminant locus of the Fermat cubic is invariant under $G$.
\end{lemma}
\begin{proof}
Making the coordinate change $x_1=a \cos \alpha$, $x_2=a \sin \alpha$, $x_3=b \cos \beta$, $x_4=b \sin \beta$ and simplifying yields the following equations for discriminant locus of the
Fermat cubic:
\begin{equation}
-2 a^3 \cos(3 \beta) \sin(
   3 \alpha) + (-1 + a^6 + 3 a^4 b^2 + 3 a^2 b^4 + b^6) \sin(3 \beta) = 0
\label{alphabetaim}
\end{equation}
\begin{multline}
1 + 4 a^6 - 4b^6 + a^{12} +b^{12} + 6 a^4b^2 - 6 a^2b^4 + 6 a^{10}b^2 + 6 a^2b^{10} + \\
   15 a^8b^4  + 15 a^4b^8 + 20 a^6b^6 +
   2b^6 \sin(6 \beta) + 
   2 a^6 \cos(6 \alpha) + \\
   4 a^3 (1 + a^6 +b^6 + 3 a^4b^2 + 3 a^2b^4 ) \cos(3 \alpha)  
 =0
\label{alphabetare}
\end{multline}   
The invariance of the discriminant locus under the rotations $\Z_3 \times \Z_6 \subset D_3 \times D_6$ is now manifest.

Under the transformation $(a \mapsto a/(a^2 + b^2), b \mapsto b/(a^2 + b^2), \alpha\mapsto-\alpha, \beta \mapsto\beta)$, these two equations are transformed to non-zero multiples of themselves. So the discriminant locus is invariant under this transformation. Similarly it is invariant under the transformation  $(a \mapsto a/(a^2 + b^2), b \mapsto b/(a^2 + b^2), \alpha\mapsto\alpha, \beta\mapsto-\beta)$. 

Together these transformations generate $G$.
\end{proof}

Although the coordinates used in the above Lemma give the simplest explanation for the symmetries of the discriminant locus, the topology seems easier to understand in the $x^\prime$ coordinates. This is because all the vertices of the cell decomposition can be viewed at the single time $x^\prime_1=0$.  It is by combining the best features of both coordinate systems that we will be able to certify the topology of the discriminant locus. Therefore, let us consider the equations for the discriminant locus in the cylindrical polar coordinates used in the previous section. The equations take form:
\begin{equation}
p_1(r,x^\prime_1,x^\prime_2) + r^6 \cos (6 \theta) +
p_2(r,x^\prime_1,x^\prime_2) \sin(3 \theta) = 0
\label{rePolar}
\end{equation}
and
\begin{equation} p_3(r, x^\prime_1,x^\prime_2) +
p_4(r,x^\prime_1,x^\prime_2) \cos(3 \theta)=0
\label{imPolar}
\end{equation}
for some polynomials $p_1$, $p_2$, $p_3$ and $p_4$. So long as
$p_4 \neq 0$
we can solve for $\cos(3 \theta)$ using equation \eqref{imPolar}.
We can then substitute the resulting value for $\cos(6 \theta)=2
\cos(3 \theta)^2 -1$ into equation \eqref{rePolar} and hence solve
for $\sin(3 \theta)$, assuming $p_2 \neq 0$.
\begin{eqnarray*}
\cos(3 \theta) & = & -\frac{p_3}{p_4} \\
\sin(3 \theta) & = & \frac{r^6 ( p_4^2 - 2p_3^2) -p_4^2 p_1 }{p_2
p_4^2}
\end{eqnarray*}
These equations have a unique solution for $\theta \in
[0,\frac{2 \pi}{3})$ so long as the sum of the squares of the
right hand sides is equal to 1. Thus away from the points where
$p_4 =0$ or $p_2=0$ we see that the discriminant locus is a
triple cover of the surface defined by the equation:
\begin{equation}
p_2^2 p_3^2 p_4^2 + (r^6 ( 1 - 2p_3^2) -p_4^2 p_1)^2 = p_2^2
p_4^4
\label{polarEqn}
\end{equation}
and the inequality
\begin{equation}
r \geq 0.
\label{boundsInequality}
\end{equation}

Excluding the cases where $p_4=0$ or $p_2=0$, we have reduced
our problem from one of considering the intersection of two
polynomials in $\R^4$ to a problem involving the zero set of a
single polynomial in $\R^3$. This is a considerable
simplification since the topology of surface defined by a single
equation is much easier to understand than a surface defined by
two equations. Having said that, the degree of equation \eqref{polarEqn} is $44$. One also needs to give careful consideration to the situation when $p_4=0$ or $p_2=0$. We omit the details, but one can, with the aid of computer algebra, prove all solutions with $p_4=0$ or $p_2=0$ in fact satisfy equation \eqref{polarEqn}.

We now make a conformal transformation of
equation \eqref{polarEqn} sending the six triple points to a
circle. We transform to new
coordinates ($(x,y,t)$) as follows:

\begin{eqnarray*}
r & = & \frac{1-x^2 - y^2 - t^2 }{x^2 + t^2 + (1 - y)^2} \\
x^{\prime}_1 & = & \frac{2 t}{x^2 + t^2 + (1 - y)^2} \\
x^{\prime}_2 & = & \frac{2 x}{x^2 +t^2 + (1-y)^2}
\end{eqnarray*}

The inequality \eqref{boundsInequality} transforms under this
coordinate change to the condition $x^2 + y^2 + t^2 <= 1$.
After this transformation, a $D_6$ symmetry does indeed become
clear.  We once again introduce cylindrical polar coordinates $y = R \sin \phi $, $u_2 = R \cos \phi$. The inequality \eqref{boundsInequality} becomes $R^2 + t^2 <= 1$.
Equation \eqref{polarEqn} takes the form:
\begin{equation}
\sum_{n \in \{6,12,18,24\} } p_n(R,t) R^{12} \cos(n \phi) 
\label{formOfDoublePolar}
\end{equation}
for appropriate polynomials $p_n$.

We write $\cos( 6 \phi) = \Phi$ for a new coordinate $\Phi \in
[-1,1]$. With this coordinate change \eqref{polarEqn} becomes a
quartic in $\Phi$ with coefficients that are polynomials in $R$
and $u_1$. The discriminant locus will be a $3 \times 12=36$
sheeted cover of the implicit surface obtained.

It turns out these coefficients are of entirely even degree in
both $R$ and $t$. So we make another simplification and introduce
new variables $(S,T)$ with $S=R^2$ and $T=t^2$ with the
conditions $S \geq 0$ and $x \geq 0$. Note that this indicates another symmetry of the discriminant locus given by the map $t->-t$.

By means of these algebraic manipulations we have now confirmed that $G$ preserves the discriminant locus and simultaneously introduced simpler coordinates.

The final change of coordinates is to write $Z=S+T$ and work
with the coordinates $(Z,\Phi,T)$. This final change of
coordinates is designed so that \eqref{boundsInequality}
transforms to the simple condition $Z \leqslant 1$.

In summary then, we have introduced new coordinates $(T,Z,\Phi)$
such that the discriminant locus is a $72$ sheeted cover of the
implicit surface defined by the inequalities:
\begin{equation}
T \geq 0,\quad 0 \leqslant Z \leqslant 1, \quad -1 \leqslant \Phi \leqslant 1
\label{quarticBoundary}
\end{equation}
and the transformed equation \eqref{polarEqn}. The transformed
equation will be of degree $4$ in $\Phi$.

We know from equation \eqref{formOfDoublePolar} that there is a
constant factor of $R^{12}$ that we can cancel. After making
this cancellation and also the cancellation of a constant
factor, one obtains the equation:

\begin{flushleft}
{\scriptsize
$
426 T^2 - 228 \Phi T^2 - 6 \Phi^2 T^2 + 3792 T^3 - 496 \Phi T^3
-
 848 \Phi^2 T^3 - 16 \Phi^3 T^3 + 7584 T^4 
 + 96 \Phi T^4 -  1056 \Phi^2 T^4 - 480 \Phi^3 T^4 + 4608 T^5 +
1536 \Phi T^5 +  1536 \Phi^2 T^5 - 1536 \Phi^3 T^5
+ 128 T^6 + 1024 \Phi T^6 +   1792 \Phi^2 T^6 - 1024 \Phi^3 T^6
+ 128 \Phi^4 T^6 + 228 T Z +  240 \Phi T Z
+ 12 \Phi^2 T Z + 3936 T^2 Z + 432 \Phi T^2 Z + 1728 \Phi^2 T^2
Z
+ 48 \Phi^3 T^2 Z + 22944 T^3 Z - 288 \Phi T^3 Z -
1056 \Phi^2 T^3 Z + 1440 \Phi^3 T^3 Z + 31104 T^4 Z - 384 \Phi
T^4 Z -
 1920 \Phi^2 T^4 Z + 1920 \Phi^3 T^4 Z + 8448 T^5 Z +
 4608 \Phi^2 T^5 Z - 768 \Phi^4 T^5 Z - 6 Z^2 - 12 \Phi Z^2 - 
 6 \Phi^2 Z^2 + 912 T Z^2 + 48 \Phi T Z^2 - 912 \Phi^2 T Z^2 -
 48 \Phi^3 T Z^2 + 13368 T^2 Z^2 + 1968 \Phi T^2 Z^2 + 
 5304 \Phi^2 T^2 Z^2 - 1440 \Phi^3 T^2 Z^2 + 55344 T^3 Z^2 -
3600 \Phi T^3 Z^2 - 3504 \Phi^2 T^3 Z^2 + 3600 \Phi^3 T^3 Z^2 + 47424 T^4 Z^2 - 5568 \Phi T^4 Z^2 - 10176 \Phi^2 T^4 Z^2 + 
 3264 \Phi^3 T^4 Z^2 + 1920 \Phi^4 T^4 Z^2 + 4608 T^5 Z^2 + 
1536 \Phi T^5 Z^2 + 1536 \Phi^2 T^5 Z^2 - 1536 \Phi^3 T^5 Z^2 + 16 \Phi Z^3 + 32 \Phi^2 Z^3 + 16 \Phi^3 Z^3 + 1872 T Z^3 - 
 1824 \Phi T Z^3 - 3216 \Phi^2 T Z^3 + 480 \Phi^3 T Z^3 + 
 27168 T^2 Z^3 + 4176 \Phi T^2 Z^3 + 6336 \Phi^2 T^2 Z^3 - 
 6960 \Phi^3 T^2 Z^3 + 74944 T^3 Z^3 - 1472 \Phi T^3 Z^3 - 
4544 \Phi^2 T^3 Z^3 + 192 \Phi^3 T^3 Z^3 - 2560 \Phi^4 T^3 Z^3 + 31104 T^4 Z^3 - 384 \Phi T^4 Z^3 - 1920 \Phi^2 T^4 Z^3 + 
 1920 \Phi^3 T^4 Z^3 + 24 Z^4 + 48 \Phi Z^4 + 24 \Phi^2 Z^4 + 
3696 T Z^4 - 1584 \Phi T Z^4 - 2160 \Phi^2 T Z^4 + 3120 \Phi^3 T
Z^4 +
 35004 T^2 Z^4 + 8808 \Phi T^2 Z^4 + 12444 \Phi^2 T^2 Z^4 - 
 4800 \Phi^3 T^2 Z^4 + 1920 \Phi^4 T^2 Z^4 + 55344 T^3 Z^4 - 
3600 \Phi T^3 Z^4 - 3504 \Phi^2 T^3 Z^4 + 3600 \Phi^3 T^3 Z^4 + 7584 T^4 Z^4 + 96 \Phi T^4 Z^4 - 1056 \Phi^2 T^4 Z^4 - 
480 \Phi^3 T^4 Z^4 - 144 \Phi Z^5 - 288 \Phi^2 Z^5 - 144 \Phi^3
Z^5 +
4248 T Z^5 - 2976 \Phi T Z^5 - 4344 \Phi^2 T Z^5 + 2112 \Phi^3 T
Z^5 -
 768 \Phi^4 T Z^5 + 27168 T^2 Z^5 + 4176 \Phi T^2 Z^5 + 
 6336 \Phi^2 T^2 Z^5 - 6960 \Phi^3 T^2 Z^5 + 22944 T^3 Z^5 - 
288 \Phi T^3 Z^5 - 1056 \Phi^2 T^3 Z^5 + 1440 \Phi^3 T^3 Z^5 +
92 Z^6 +
184 \Phi Z^6 + 220 \Phi^2 Z^6 + 256 \Phi^3 Z^6 + 128 \Phi^4 Z^6
+
3696 T Z^6 - 1584 \Phi T Z^6 - 2160 \Phi^2 T Z^6 + 3120 \Phi^3 T
Z^6 +
 13368 T^2 Z^6 + 1968 \Phi T^2 Z^6 + 5304 \Phi^2 T^2 Z^6 - 
 1440 \Phi^3 T^2 Z^6 + 3792 T^3 Z^6 - 496 \Phi T^3 Z^6 - 
848 \Phi^2 T^3 Z^6 - 16 \Phi^3 T^3 Z^6 - 144 \Phi Z^7 - 288
\Phi^2 Z^7 -
144 \Phi^3 Z^7 + 1872 T Z^7 - 1824 \Phi T Z^7 - 3216 \Phi^2 T
Z^7 +
 480 \Phi^3 T Z^7 + 3936 T^2 Z^7 + 432 \Phi T^2 Z^7 + 
1728 \Phi^2 T^2 Z^7 + 48 \Phi^3 T^2 Z^7 + 24 Z^8 + 48 \Phi Z^8 +24 \Phi^2 Z^8 + 912 T Z^8 + 48 \Phi T Z^8 - 912 \Phi^2 T Z^8 -48 \Phi^3 T Z^8 + 426 T^2 Z^8 - 228 \Phi T^2 Z^8 - 6 \Phi^2 T^2
Z^8 +
 16 \Phi Z^9 + 32 \Phi^2 Z^9 + 16 \Phi^3 Z^9 + 228 T Z^9 + 
240 \Phi T Z^9 + 12 \Phi^2 T Z^9 - 6 Z^{10} - 12 \Phi Z^{10} - 6
\Phi^2 Z^{10} = 0.$
}
\end{flushleft}
\begin{equation}
\label{quartic}
\end{equation}
Ugly though this formula is, from a computational complexity viewpoint it is a simplification of equations \eqref{alphabetaim} and \eqref{alphabetare}. We started with two
degree $12$ polynomials in $4$ variables and have reduced the
problem to a single degree $12$ polynomial in $3$ variables.
Moreover, the polynomial is of degree $4$ in one of those
variables.

\begin{figure}[ht]
\begin{centering}
\includegraphics{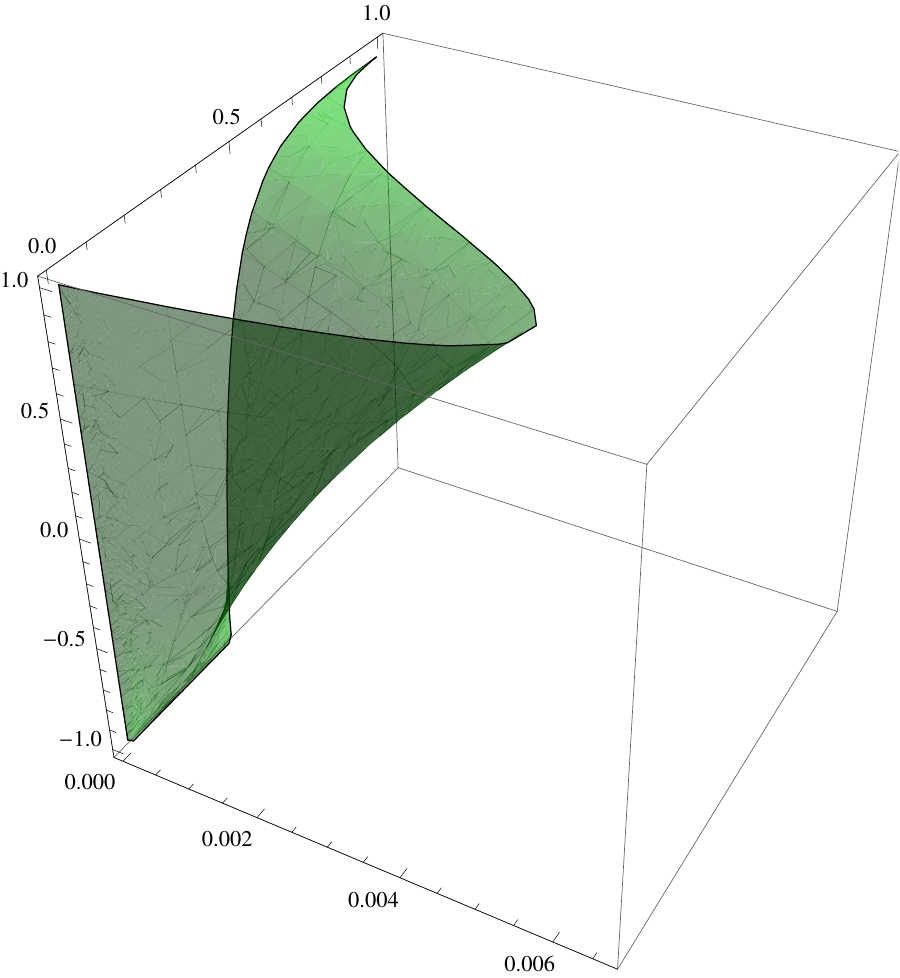}
\end{centering}
\caption{The surface defined by equations \eqref{quarticBoundary}
and \eqref{quartic}.}
\label{fermatFundamentalDomain}
\end{figure}

A plot of the surface defined by these conditions is shown in
Figure \ref{fermatFundamentalDomain}. Since we introduced $S=R^2$ where $R$ is a radial coordinate, one needs to identify certain points with $S=0$. Thus a fundamental domain of $G$ is given by  an appropriate quotient of Figure \ref{fermatFundamentalDomain}.

We have identified a group of order $72$ acting on the
discriminant locus. Figure \ref{fermatFundamentalDomain}
represents a fundamental domain of this group action. 

The cell decomposition used in the proof of Proposition \ref{topologyCalculation} only had $36$ faces. To obtain the $72$ faces associated with $G$, one divides each face in the cell decomposition used earlier into two pieces by introducing new edges from the twistor lines $T_j$ to the triple points $V_i$. With this understood, the proof of Proposition \ref{topologyCalculation} amounts to a computation of the topology of the tiling of the covering space with copies of this fundamental domain.

\section{Cylindrical algebraic decomposition of the fundamental domain}
\label{cylindricalDecompositionSection}

We now wish to:
\begin{enumerate}
\item Identify the points on Figure
\ref{fermatFundamentalDomain} that correspond to points with non
trivial stabilizer under $G$ and confirm that they split into
vertices and edges homeomorphic to $[0,1]$ as expected.
\item Show that the Figure \ref{fermatFundamentalDomain}, with the appropriate
quotients when $S=0$, is
homeomorphic to the closed unit disc with boundary given by the
points with non trivial stabilizer under $G$.
\end{enumerate}
If we can complete these tasks we will have rigorously identified a cell decomposition of the discriminant locus.

We begin with the first part. By construction of our coordinate
system, the points with non zero stabilizer are points with
$S=0$, $\Phi=-1$, $T=0$, $\Phi=1$. Simplifying equation
\eqref{quartic} for each of the first three cases one obtains the
following simple equations:
\begin{equation}
8 T^2 (3 + 10 T + 3 T^2)^4 = 0
\label{stabilizer1}
\end{equation}
\begin{equation}
8 T^2 (3 + 3 S^2 + 10 T + 3 T^2 + 6 S (1 + T))^4 = 0
\label{stabilizer2}
\end{equation}
\begin{multline}
2 (1 + \Phi)^2 Z^2 (-3 + 12 Z^2 + 46 Z^4 + 64 \Phi^2 Z^4 \\
 + 12 Z^6 -   3 Z^8 
+ 8 \Phi (Z - 9 Z^3 - 9 Z^5 + Z^7))=0
\label{stabilizer3}
\end{multline}

The first equation has solutions $T=0$, $\Phi \in [-1,1]$. The
second equation has the unique solution $S=T=0$. The third is
quadratic in $\Phi$ so is also easily understood. All of these
solutions have in common the fact that $T=0$. So referring back
to Figure \ref{fermatDiscriminantLocus} these solutions
correspond to the edges shown at time $x^\prime_1=0$.

The solutions of the equation $\Phi=1$ correspond to edges
used to split the cell decomposition consisting of $36$ faces into one of $72$ faces. The equation in this case is a little
more complex, but does factor into two cubics in $T$:
\begin{multline}
\mbox{\scriptsize $8(1 + 12 T + 24 T^2 + 16 T^3 + 36 T Z + 48 T^2 Z - 6 Z^2 + 36 T
Z^2 +
24 T^2 Z^2 + 8 Z^3 + 60 T Z^3 - 3 Z^4) \times$} \\
\mbox{\scriptsize $(24 T^2 + 16
T^3 + 60 T Z +
48 T^2 Z - 3 Z^2 + 36 T Z^2 + 24 T^2 Z^2 + 8 Z^3 + 36 T Z^3 -   6 Z^4 + 12 T Z^4 + Z^6)= 0$}.
\label{stabilizer4}
\end{multline}

We need to show that the curve defined by this equation 
together with the conditions $0 \leqslant Z \leqslant 1$ and $0 \leqslant
T \leqslant 1$ is homeomorphic to $[0,1]$. One could do this with
one's bare hands, but we will instead discuss how one can use the
``cylindrical algebraic decomposition'' algorithm.

Cylindrical algebraic decomposition is a foundational algorithm
in
computational real algebraic geometry. It is an algorithm that
allows one to decompose a semi-algebraic set (that is a real set
defined by polynomial equalities and inequalities) into pieces
called cylindrical sets which are all homeomorphic to either
$\{0\}$ or $\R^i$ for some $i$. In effect, it computes a CW complex that is guaranteed to be homeomorphic to the semi-algebraic set. See \cite{basu} for more information. 

There is a catch: the running time of cylindrical algebraic
decomposition is of the order $(s d)^{c^{n-1}}$ where $s$ is the
number of equations and inequations used to define the semi-algebraic
set, $n$ is the number of variables, $d$ is the
maximum degree of the polynomials and $c$ is a constant. This
doubly exponential running time means that cylindrical algebraic
decomposition is only effective for rather small problems.

\begin{figure}[ht]
\begin{centering}
\includegraphics{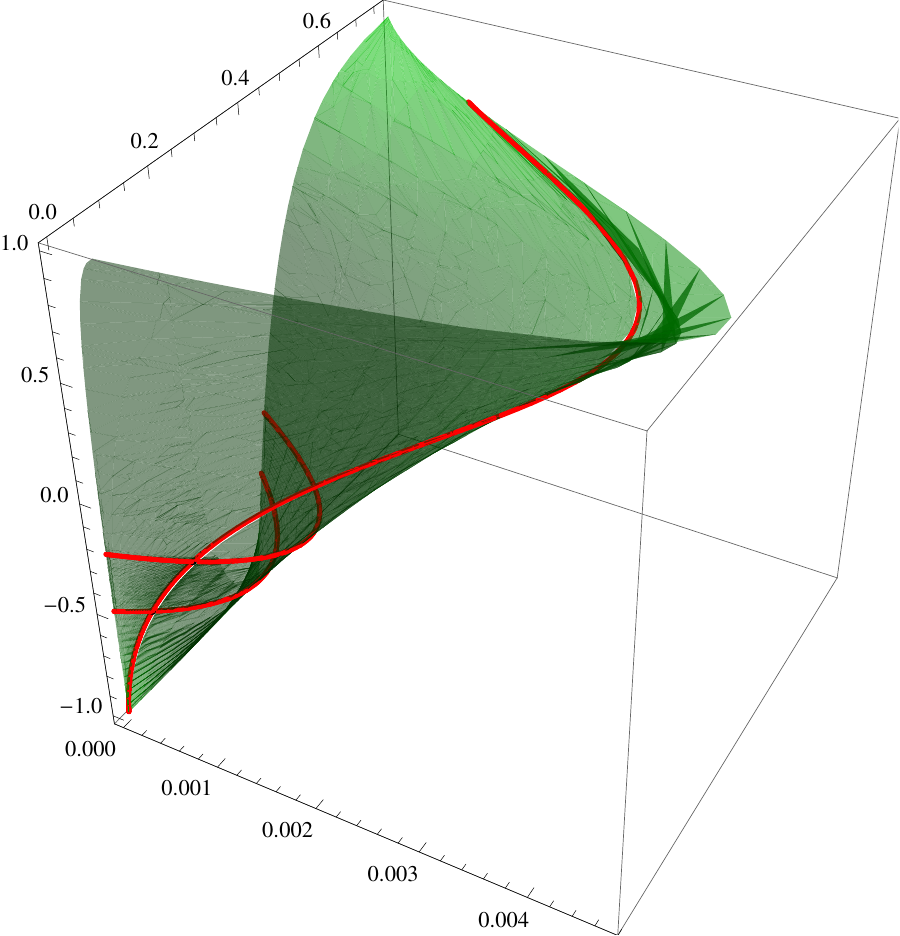}
\end{centering}
\caption{Cylindrical algebraic decomposition of the fundamental
domain}
\label{cad}
\end{figure}

Fortunately the semi-algebraic set defined by \eqref{stabilizer4}
and $0 \leqslant Z \leqslant 1$ and $0 \leqslant T \leqslant 1$ is appropriately
small and one can perform the cylindrical algebraic
decomposition quickly. The algorithm works inductively by using
resultants to project the equations along each of the coordinate
axes to obtain lower dimensional equations. Thus the specific
cylindrical decomposition depends upon the semi-algebraic set and the ordering of the coordinate axes. For this case, if one chooses the order $(Z,T)$ one obtains a decomposition into the following three cylindrical sets:
\begin{align*}
&\{(0,0)\}, \\
&\{(1,0)\}, \, \\
&\{(Z,T) : 0<Z<1 \hbox{ and } T=\hbox{root}_\lambda(
-3 Z^2 + 8 Z^3 - 6 Z^4 + Z^6 + \\
&\quad(60 Z + 36 Z^2 + 36 Z^3 + 12 Z^4) \lambda +
 (24 + 48 Z +  24 Z^2) \lambda^2 + 16 \lambda^3, 3) \}
\end{align*}
Here we are using the notation $\hbox{root}_{\lambda}(
p(\lambda), n)$ to denote the $n^{\mathrm{th}}$ real root of a
polynomial $p$ in the variable $\lambda$. We conclude that the
semi-algebraic set defined by equation \eqref{stabilizer4} and $0 \leqslant Z
\leqslant 1$ and $0 \leqslant T \leqslant 1$ is homeomorphic to $[0,1]$ as
required.

The final ingredient required to complete the proof is to show
that the semi-algebraic set defined by equations \eqref{quarticBoundary}
and \eqref{quartic} together with the appropriate quotienting when $S=0$
is homeomorphic to the closed disc and that its
boundary consists
of the stabilizer just identified. Using Mathematica we do this by computing the cylindrical algebraic decomposition with respect to the variable ordering $\Phi$, $Z$, $T$. The computation takes several minutes to run.
Notice that our many simplifications of the equations defining the discriminant locus are crucial to achieving this. For example simply failing to introduce the
coordinate $Z=S+T$ is enough to prevent the algorithm running to completion on our computers.

Since the cylindrical algebraic decomposition in this case is far too lengthy to print out, we have simply plotted the
decomposition in Figure \ref{cad}. It consists of six faces. As we saw in the example above, the specific formulae in the cylindrical algebraic decomposition are irrelevant for our purposes. All that matters is the topology implied by the decomposition. Thus the picture in Figure \ref{cad} summarises all the key information we need about the cylindrical algebraic decomposition. We conclude that our fundamental domain is homeomorphic to 6 closed discs glued as indicated. Hence it is homeomorphic to a disc. Thus the cylindrical algebraic decomposition provides rigorous confirmation of what was visually obvious in Figure \ref{fermatFundamentalDomain}.

We summarize our findings which confirm Proposition \ref{topologyCalculation}:

\begin{theorem}
The discriminant locus of the Fermat cubic has topology represented by the graph:
\[
\begin{tikzpicture}[scale=1]
\node (n) at (1,1) {$\Sigma_4$} ;
\node [circle, draw, fill=white,inner sep=0pt, minimum width=4pt] (a) at (0,0) {};
\node [circle, draw, fill=white,inner sep=0pt, minimum width=4pt] (b) at (1,0) {};
\node [circle, draw, fill=white,inner sep=0pt, minimum width=4pt] (c) at (2,0) {};
\draw [double] (a) -- (n); 
\draw [double] (b) -- (n); 
\draw [double] (c) -- (n); 
\end{tikzpicture}.
\]
It has conformal symmetry group $D_3 \times D_6$.
\end{theorem}

In \cite{armstrongSalamonSigma} a less formal computation is given for the discriminant locus of the cubic:
\[ z_1^2 z_4 + z_4^2 z_1 + z_2^2 z_3 + z_3^2 z_2 = 0. \]
This is called the {\em transformed Fermat cubic} since it is conformally equivalent but not projectively equivalent to the Fermat cubic. According to the visual examination of the discriminant locus carried out in \cite{armstrongSalamonSigma} this has topology given by the graph:
\[
\begin{tikzpicture}[scale=1]
\node (n) at (2,1) {$\Sigma_1$} ;
\node (m) at (2,-1) {$\Sigma_1$} ;
\node [circle, draw, fill=white,inner sep=0pt, minimum width=4pt] (a) at (0,0) {};
\node [circle, draw, fill=white,inner sep=0pt, minimum width=4pt] (b) at (1,0) {};
\node [circle, draw, fill=white,inner sep=0pt, minimum width=4pt] (c) at (2,0) {};
\node [circle, draw, fill=white,inner sep=0pt, minimum width=4pt] (d) at (3,0) {};
\node [circle, draw, fill=white,inner sep=0pt, minimum width=4pt] (e) at (4,0) {};
\draw [-] (a) -- (n); 
\draw [-] (b) -- (n); 
\draw [-] (c) -- (n); 
\draw [-] (d) -- (n); 
\draw [-] (e) -- (n); 
\draw [-] (a) -- (m); 
\draw [-] (b) -- (m); 
\draw [-] (c) -- (m); 
\draw [-] (d) -- (m); 
\draw [-] (e) -- (m); 
\end{tikzpicture}.
\]
Unfortunately we are not able to reduce the problem to a computationally feasible cylindrical algebraic decomposition in this case, so the result is not fully rigorous. The five topological singularities in this graph correspond to five twistor lines on the transformed Fermat cubic.

It seems that the local topology at the twistor lines is the same for both the rotated Fermat cubic and the Fermat cubic. Likewise the local topology at the  triple points is the same in both cases. It would be very interesting if one could catalogue in full the possible topologies at twistor lines and triple points on a non-singular cubic surface and thereby develop a more extensive theory of twistor cubic surfaces.

\bibliographystyle{alpha}	
\bibliography{fermat}

\end{document}